\documentclass{amsart}
\usepackage{amsmath}
\usepackage{amsthm}
\usepackage{amsfonts}
\usepackage{amssymb}
\usepackage{amscd}
\usepackage{stmaryrd}
\usepackage{hyperref}
\usepackage{enumitem}
\usepackage{graphicx}

\numberwithin{equation}{section}

\theoremstyle{plain}
\newtheorem{thm}[equation]{Theorem}
\newtheorem{lemma}[equation]{Lemma}
\newtheorem{cor}[equation]{Corollary}
\newtheorem{prop}[equation]{Proposition}
\newtheorem{example}[equation]{Example}

\theoremstyle{definition}

\newtheorem{question}[equation]{Question}
\newtheorem{remark}[equation]{Remark}

\newcommand{\Z}{\ensuremath \mathbb{Z}}

\newcommand{\M}{\ensuremath \mathbb{M}}
\newcommand{\isom}{\ensuremath \cong}
\newcommand{\ann}{\ensuremath{\mbox{\rm ann}}}

\usepackage{calc}
\makeatletter
\newlength{\dolz}\newlength{\dolzp}\newlength{\dolzr}
\newcounter{i}\newcounter{j}
\newcommand{\urmatrix}[6]{
\setlength{\dolzp}{0pt}
\@for\tmp:=#3\do{\setlength{\dolzr}{\widthof{$\scriptstyle\tmp$}}
\ifdim\dolzp<\dolzr\setlength{\dolzp}{\dolzr}\fi}
\setlength{\dolz}{\totalheightof{$\displaystyle\begin{array}{#1}#2\end{array}$}/2}
\setcounter{i}{0}\setlength{\dolzr}{0pt}
\@for\tmp:=#3\do{\stepcounter{i}\setcounter{j}{0}
\@for\tmr:=#5\do{\stepcounter{j}\ifnum\value{j}=\value{i}
\setlength{\dolz}{\dolz-\totalheightof{$\displaystyle\begin{array}{c}\tmr\end{array}$}}
\setlength{\dolzr}{\dolzp-\widthof{$\scriptstyle\tmp$}}
\raisebox{\dolz}{\hspace{\dolzr}$\scriptstyle\tmp$}
\hspace{-\dolzp}\fi}}
\hspace{\dolzp}\!
\stackrel{\setcounter{i}{0}
\@for\tmp:=#4,\do{\stepcounter{i}\setcounter{j}{0}
\@for\tmr:=#6\do{\stepcounter{j}\ifnum\value{j}=\value{i}
\setlength{\dolz}{\widthof{$\tmr$}+2\arraycolsep}\hspace{\dolz}\fi}
\setlength{\dolz}{0pt-\widthof{$\scriptstyle\tmp$}/2}
\hspace{\dolz}\tmp\hspace{\dolz}}}{\left(\begin{array}{#1}#2\end{array}\right)}}

\begin{document}

\title[Connections between unit-regularity, regularity, cleanness, and strong cleanness]{Connections between unit-regularity, regularity, cleanness, and strong cleanness of elements and rings}

\author{Pace P. Nielsen}
\address{Department of Mathematics, Brigham Young University, Provo, UT 84602, USA}
\email{pace@math.byu.edu}

\author{Janez \v{S}ter}
\address{Faculty of Mechanical Engineering, University of Ljublijana, A\v{s}ker\v{c}eva 6, 1000 Ljublijana, Slovenia}
\email{janez.ster@fs.uni-lj.si}


\begin{abstract}
We construct an example of a unit-regular ring which is not strongly clean, answering an open question of Nicholson.  We also characterize clean matrices with a zero column, and this allows us to describe an interesting connection between unit-regular elements and clean elements.  It is also proven that given an element $a$ in a ring $R$, if $a,a^2,\ldots, a^k$ are all regular elements in $R$ (for some $k\geq 1$), then there exists $w\in R$ such that $a^{i}w^{i}a^{i}=a^{i}$ for $1\leq i\leq k$, and a similar statement holds for unit-regular elements.  The paper ends with a large number of examples elucidating further connections (and disconnections) between cleanliness, regularity, and unit-regularity.
\end{abstract}

\maketitle

\section{Introduction}

Regular rings were defined by von Neumann in his study of continuous geometries, and have become a staple of noncommutative ring theory due to their simple definition and connection to decomposition theory.  An element $a\in R$ is said to be (von Neumann) \emph{regular} if there exists some $r\in R$ with $ara=a$.  We denote the set of all regular elements of $R$ by $\operatorname{reg}(R)$, and if $\operatorname{reg}(R)=R$ one says that $R$ is a \emph{regular ring}.  The element $r$ is called an \emph{inner inverse} for $a$, and it need not be unique.

As defined by Ehrlich \cite{EhrlichFirst}, an element $a\in R$ is \emph{unit-regular} if it has a unit inner inverse $u\in U(R)$. The set of unit-regular elements is denoted by $\operatorname{ureg}(R)$, and if $\operatorname{ureg}(R)=R$ we say $R$ is a \emph{unit-regular ring}.  Examples include all semisimple rings and all commutative regular rings.  However, there are regular rings which are not unit-regular.

The utility of these definitions is most easily seen by the following classical result.

\begin{lemma}\label{Lemma:RegDirect}
Let $\,_kM$ be a left $k$-module for some ring $k$, and let $a\in R:=\operatorname{End}(M)$.  We have $a\in \operatorname{reg}(R)$ if and only if $\ker(a)$ and ${\rm im}(a)$ are direct summands of $M$.  Moreover, $a\in \operatorname{ureg}(R)$ if and only if $a\in \operatorname{reg}(R)$ and $\ker(a)\isom {\rm coker}(a)$.
\end{lemma}
\begin{proof}
See Exercises 4.14A$_1$ and 4.14C from \cite{LamExercises}, and the comments following them.  The second exercise appears as part of \cite[Theorem 1]{Ehrlich}, with the additional hypothesis that the entire ring $R$ is regular.
\end{proof}

Another generalization of regular rings arose in the work of Warfield \cite{Warfield}, based on earlier work of Crawley and J{\'o}nsson on direct sum decomposition theory \cite{CJ}.  Nicholson \cite{Nicholson77} showed that Warfield's rings are exactly those for which idempotents lift modulo every one-sided ideal, and they are called the \emph{exchange rings}.

Many exchange rings satisfy additional properties, which Nicholson described in \cite{Nicholson77} and \cite{NicholsonStrongClean}.  Following that work, we say that an element $a\in R$ is \emph{clean} if $a=e+u$ for some idempotent $e^2=e$ and some unit $u\in U(R)$.  The additive decomposition $a=e+u$ is sometimes called a \emph{clean decomposition}.  If $a$ has some clean decomposition $a=e+u$ where additionally $e$ and $u$ commute, then we say $a$ is \emph{strongly clean}.  When every element of $R$ is (strongly) clean, then one says that $R$ is a \emph{$($strongly$)$ clean ring}.  Surprisingly, most natural examples of exchange rings are clean rings, including all endomorphism rings of continuous modules \cite{CKLNZ}, all commutative (and more generally abelian) exchange rings \cite[Proposition 1.8]{Nicholson77}, and all strongly $\pi$-regular rings \cite[Theorem 1]{NicholsonStrongClean}. (The last two classes are also strongly clean.)

Bergman has constructed a regular (hence exchange) ring which is not clean; see \cite[p.\ 4746]{CamilloYu} and \cite[Example 3.1]{SterCorner}.  All other examples of non-clean exchange rings seem to be based on this single example.  Remarkably,
unit-regularity of \emph{rings} does imply cleanness by \cite[Theorem 5]{CamilloYu} and \cite[Theorem 1]{CK}.

This raises the question of whether every unit-regular ring is strongly clean, which appears as the fourth of five questions in \cite{NicholsonStrongClean}.  In Section \ref{Section:StrongClean} we answer this question in the negative by constructing a unit-regular ring which is not strongly clean. (Thus, of the five questions from \cite{NicholsonStrongClean}, only Questions 1 and 2 remain open.)  It is also well-known that unit-regular \emph{elements} are not always clean.  We clarify this relationship in Section \ref{Section:Elements} by showing that there is a natural element-wise extension of unit-regularity which is equivalent to an enhanced form of cleanness, proven in Theorem \ref{correspondence}.  One of the conditions of that theorem leads us, in Section \ref{Section:Powers}, to consider powers of regular elements.  We prove that if $a,a^2,\ldots, a^k$ are all regular in a ring $R$, then there exists some $w\in R$ such that $w,w^2,\ldots, w^k$ are corresponding inner inverses.  A similar statement about unit-regular elements also holds.  We finish by proving some natural limitations on these theorems, and providing many examples.

Rings in this paper are associative and have $1$, but are not necessarily commutative.  Modules are unital, and endomorphisms will be written on the side opposite the scalars.  The set of units in a ring $R$ will be denoted (as we did above) by $U(R)$, while the set of idempotents is $\operatorname{idem}(R)$.

\section{Unit-regular rings are not always strongly clean}\label{Section:StrongClean}

This section is devoted to constructing a unit-regular ring which is not strongly clean.  The ring we will create is isomorphic to a complicated example of Bergman found in \cite[Example 5.12]{Goodearl}, given there as an instance of a unit-regular ring with a regular subring which is not unit-regular.  Instead of following the steps taken there to define the ring as a set of certain endomorphisms, we find it easier to view the ring as a set of infinite $\Z\times \Z$ matrices, so we will form the ring along independent lines.

Let $F$ be a field.  Let $F(\!(t)\!)$ denote the field of formal Laurent series over $F$.  We define $R$ to be the set of all $\mathbb{Z}\times\mathbb{Z}$ matrices $A=(a_{i,j})_{i,j\in\mathbb{Z}}$ over $F$ such that there exist $m,n\in\mathbb{Z}$ and $f(t)=\sum_{k=k_0}^\infty a_kt^k\in F(\!(t)\!)$ with the following properties:
\begin{enumerate}[label=(R\arabic*)]
\item \label{frcond1}
if $i\geq m$ or $j<n$ then $a_{i,j}=a_{j-i}$ (where we set $a_k=0$ if $k<k_0$),
\item
\label{frcond}
the submatrix $A_0=(a_{i,j})_{i<m,j\geq n}$ has finite rank (i.e., it has only finitely many linearly independent columns, or equivalently, finitely many linearly independent rows).
\end{enumerate}
Thus, the matrix $A\in R$ in the above definition is of the following form:
\[
\def\arraystretch{1.5}
\urmatrix{ccc|ccc}
{\ddots&\vdots&&&&\\
\ddots&a_{n-m}&\vdots&\multicolumn{3}{c}{A_0}\\
&a_{n-m-1}&a_{n-m}&&&\\
\hline \reflectbox{$\ddots$}&&a_{n-m-1}&a_{n-m}&\cdots&\\
&\reflectbox{$\ddots$}&&a_{n-m-1}&a_{n-m}&\cdots\\
&&\reflectbox{$\ddots$}&&\ddots&\ddots}
{,,m,,}{,,n,,}
{\vdots,\vdots,a_{n-m-1},\reflectbox{$\ddots$},\reflectbox{$\ddots$},\reflectbox{$\ddots$}}
{\ddots,a_{n-m-1},a_{n-m-1},a_{n-m-1},a_{n-m},\ddots}
\]
The indices of the entries $a_{i,j}$ in this matrix increase when we move down or right (respectively). The lower left corner of $A_0$ is at position $(m-1,n)$.

The small letter $m$ on the left of the horizontal line indicates that the row directly below this line is the $m$th row of $A$ (i.e.\ the row $(a_{m,j})_{j\in \Z}$).  Similarly, the small letter $n$ above the vertical line indicates that the column directly to the right of the vertical line is the $n$th column of $A$ (i.e.\ $(a_{i,n})_{i\in \Z}$). In subsequent computations, we will often use these small letters since they will simplify notation.

Note that in the definition of $R$ given above, we may always increase $m$ or decrease $n$ by any finite number, because these changes
do not affect the finite rank condition on $A_0$. Alternatively, the condition \ref{frcond} in the above definition may be replaced by simply saying that $(a_{i,j})_{i<0,j\geq 0}$ has finite rank, or just that $(a_{i,j})_{i<p,j\geq q}$ has finite rank for \emph{some} (arbitrary) $p,q\in\mathbb{Z}$.

There is one more way of characterizing the finite rank condition \ref{frcond}.  Note that if the matrix $A_0=(a_{i,j})_{i<m,j\geq n}$ has finite rank, then there exists $c\geq 1$ such that each row in $A_0$ is a linear combination of the bottom $c$ rows of $A_0$. Writing a decomposition
\[
\def\arraystretch{1.1}\arraycolsep=0.5cm
A_0=\left(\begin{array}{c}A_2\\\hline A_1\end{array}\right),
\]
with $A_1$ composed of the last $c$ rows of $A_0$, this means that we can factor $A_2=YA_1$ for a suitable infinite matrix $Y$ (with $c$ columns). Moreover, since $A_1$ has only finitely many independent columns, we can find $d\geq 1$ such that $A_1$ decomposes as
\[
A_1=\left(\begin{array}{c|c}X_0&X_0Z\end{array}\right)
\]
where $X_0$ contains the first $d$ columns of $A_1$ and $Z$ is an infinite matrix with $d$ rows. Putting this together, we get a decomposition of $A_0$
\begin{equation}\label{finrankdec}
A_0=\left(\begin{array}{cc}YX_0&YX_0Z\\X_0&X_0Z\end{array}\right).
\end{equation}
It is clear that such a decomposition is also sufficient for $A_0$ to have finite rank.

Let us now proceed with our construction. We endow the set $R$ with the usual matrix addition and multiplication. It is clear that addition on $R$ is well-defined. Let us briefly sketch the proof of why multiplication is also well-defined. First, it is clear that if $A,B\in R$ then $AB$ is well-defined as a $\mathbb{Z}\times\mathbb{Z}$ matrix over $F$, because rows of $A$ are bounded from the left and columns of $B$ are bounded from below. Next, if the Laurent series $f(t),g(t)\in F(\!(t)\!)$ correspond to $A$ and $B$ respectively, then a straightforward computation shows that $f(t)g(t)\in F(\!(t)\!)$ corresponds to $AB$ in the sense that condition \ref{frcond1} holds for that Laurent series (for sufficiently large $m$ and sufficiently small $n$).  And finally, to see that $AB$ satisfies the finite rank condition, decompose
\[
\def\arraystretch{1.2}
A=\urmatrix{c|c}{A_{11}&A_{12}\\\hline A_{21}&A_{22}}{0}{0}{A_{12},A_{22}}{A_{21},A_{22}}
\qquad\textnormal{and}\qquad
B=\urmatrix{c|c}{B_{11}&B_{12}\\\hline B_{21}&B_{22}}{0}{0}{B_{12},B_{22}}{B_{21},B_{22}}.
\]
Since $A_{12}$ and $B_{12}$ have finite ranks, $A_{11}B_{12}$ and $A_{12}B_{22}$ also have finite ranks, and hence the upper right corner of $AB$, which is $A_{11}B_{12}+A_{12}B_{22}$, also has finite rank.

By checking the ring axioms directly, or appealing to the fact that they hold in the (over)ring of column finite matrices, this proves $R$ is a ring.  For every $A\in R$, we denote by $\psi(A)\in F(\!(t)\!)$ the Laurent series that corresponds to $A$. Note that $\psi:R\to F(\!(t)\!)$ is a surjective ring homomorphism.

\begin{prop}
The ring $R$ is unit-regular.
\end{prop}
\begin{proof}
Let $A=(a_{i,j})\in R$. Fix $m,n\in\mathbb{Z}$ and $f(t)=\sum_{k\geq k_0}a_kt^k$ such that $a_{i,j}=a_{j-i}$ if $i\geq m$ or $j<n$.  First, suppose that $f(t)\neq 0$. We may assume that $a_{k_0}\neq 0$ and, after increasing $m$, that $m+k_0>n$. Decompose $A$ as
\[
\def\arraystretch{1.2}
A=\ \urmatrix{c|c|c}{S&Y&W\\\hline 0&X_0&Z\\\hline 0&0&T}{n-k_0,m}{n,m+k_0}{S,X_0,0}{S,X_0,W}.
\]
Observe that the matrices $S$ and $T$ are upper triangular, constant on diagonals, and nonzero on the main diagonal. Therefore, they are invertible. In fact, the inverses $S^{-1}$ and $T^{-1}$ are also upper triangular and constant on diagonals, with entries that are precisely the coefficients of $f(t)^{-1}$ (with the leading term on the main diagonal). Moreover, since any finite matrix ring over a field is unit-regular, there exists an invertible matrix $U_0\in \M_{m+k_0-n}(F)$ such that $X_0U_0X_0=X_0$. Now define
\[
\def\arraystretch{1.2}
U:=\
\urmatrix{c|c|c}{S^{-1}&-S^{-1}YU_0&S^{-1}(YU_0Z-W)T^{-1}\\\hline 0&U_0&-U_0ZT^{-1}\\\hline 0&0&T^{-1}}{n,m+k_0}{n-k_0,m}{-S^{-1}YU_0,-U_0ZT^{-1},T^{-1}}{S^{-1},-S^{-1}YU_0,S^{-1}(YU_0Z-W)T^{-1}}.
\]
Since $U_0$ and $W$ have finite ranks, $S^{-1}(YU_0Z-W)T^{-1}$ has finite rank, hence $U\in R$.  Note that $U$ is invertible, with the inverse
\[
\def\arraystretch{1.2}
U^{-1}=\
\urmatrix{c|c|c}{S&Y&W\\\hline 0&U_0^{-1}&Z\\\hline 0&0&T}{n-k_0,m}{n,m+k_0}{S,U_0^{-1},0}{S,U_0^{-1},W}.
\]
One also easily verifies that $AUA=A$, which completes the proof in this case.

Now suppose that $f(t)=0$. Since $A_0=(a_{i,j})_{i<m,j\geq n}$ has finite rank, we can apply decomposition (\ref{finrankdec}), so that $A$ becomes
\begin{equation}\label{Eq:Psi0}
\def\arraystretch{1.2}
A=\
\urmatrix{c|c|c}{0&YX_0&YX_0Z\\\hline 0&X_0&X_0Z\\\hline 0&0&0}{m-c,m}{n,n+d}{X_0,X_0,0}{0,YX_0,YX_0Z}.
\end{equation}
We may assume that $c=d$ by increasing one or the other as needed. Taking an invertible $U_0\in \M_c(F)$ with $X_0U_0X_0=X_0$, one now easily verifies that
\[
\def\arraystretch{1.2}
U:=\ \urmatrix{c|c|c}{I&0&0\\\hline 0&U_0&0\\\hline 0&0&I}{n,n+c}{m-c,m}{I,U_0,I}{0,U_0,0}
\]
is invertible in $R$, with inverse
\[
\def\arraystretch{1.2}
U^{-1}=\ \urmatrix{c|c|c}{I&0&0\\\hline 0&U_0^{-1}&0\\\hline 0&0&I}{m-c,m}{n,n+c}{I,U_0^{-1},I}{0,U_0^{-1},0},
\]
and that $AUA=A$. This completes the proof.
\end{proof}

\begin{thm}
The ring $R$ is not strongly clean.
\end{thm}
\begin{proof}
Let
\[
\def\arraystretch{1.2}
A=\ \urmatrix{c|c|c|c}{I&0&0&0\\\hline 0&1&1&0\\\hline 0&0&0&0\\\hline
0&0&0&I}{-1,0,1}{-2,-1,0}{0,0,0,0}{0,0,0,0}\in R.
\]
That is, $A=(a_{i,j})$ where $a_{i,i-1}=1$ for all $i\in\Z\setminus\{0\}$, $a_{-1,-1}=1$, and $a_{i,j}=0$ for all other
pairs $(i,j)$. We will prove that $A$ is not strongly clean in $R$.

Let $E=(e_{i,j})\in R$ be any idempotent such that $AE=EA$.  We have two possible cases, $\psi(E)=0$ or $\psi(E)=1$.  First, suppose that $\psi(E)=0$. We claim that in this case $e_{i,j}=0$ for all $i\geq 0$ and $j\in\Z$. Assume contrapositively that there are some $i\geq 0$ and $j\in\Z$ such that $e_{i,j}\neq 0$. Suppose that $i$ is the largest integer with this property, so that $e_{i',j'}=0$ for all $i'>i$ and $j'\in\Z$.  In this case, the $(i+1,j)$ entry of $EA$ is $0$, while the $(i+1,j)$ entry of $AE$ is $e_{i,j}\neq 0$, and so $EA\ne AE$, proving our claim.  But this means that the $0$th row of $A-E$ is all zeros, and hence $A-E$ is not a unit.

Now suppose that $\psi(E)=1$. Then $E'=(e'_{i,j})=1-E$ satisfies $\psi(E')=0$ and $AE'=E'A$. Similarly as above, we can show that this implies $e'_{i,j}=0$ for all $i\in\Z$ and $j<0$. Accordingly, $A-E=A-1+E'$ has the $(-1)$th column all zeros and hence it cannot be a unit. This proves that $A$ is not strongly clean.
\end{proof}

\begin{remark}
(1) In the proof of this theorem, the only fact we needed about idempotents $E\in \operatorname{idem}(R)$ is that $\psi(E)\in \{0,1\}$.

(2) In the above proof, one could prove $A-E$ is not a unit by merely assuming that $AE=EAE$, rather than $AE=EA$. In fact, if $AE=EAE$ and $\psi(E)=0$, then the same argument as above shows that $E=(e_{i,j})$ must satisfy $e_{i,j}=0$ for all $i\geq 0$ and $j\in\Z$. Hence $A-E$ cannot be a unit as it has a zero $0$th row. Similarly, if $AE=EAE$ and $\psi(E)=1$, then $E'=(e'_{i,j})=1-E$ satisfies $\psi(E')=0$ and $E'A=E'AE'$, from which we obtain $e'_{i,j}=0$ for all $i\in\Z$ and $j<0$. Hence $A-E=A-1+E'$ cannot be a unit as it has a zero $(-1)$th column.  This shows that $R$ is not even \textit{capably clean} in the terminology of \cite{CDN}. In that paper it is shown that one-sided continuous, regular rings are capably clean. In agreement with this fact, one can check directly that the ring $R$ we have defined is not left or right continuous.
\end{remark}

It may be interesting to note that we can view $R$ as a subring of the endomorphism ring of $F(\!(t)\!)$.  Given an element $f(t)=\sum_{k\geq k_0}a_k t^k\in F(\!(t)\!)$, we can identify the Laurent series as an infinite row vector $(a_k)_{k\in \Z}$, which has zero entries for sufficiently negative indices.  The ring $R$ acts on such vectors by right multiplication.  In this way one can identify the ring $R$ we have constructed with the ring $Q$ of endomorphisms given in Example 5.12 from \cite{Goodearl}.  However, we will not give the formal details of this identification, since that example is quite involved, and we don't need this fact.

\section{Element-wise connections between unit-regularity and clean decompositions}\label{Section:Elements}

As proved in \cite{CamilloYu} unit-regular rings are clean rings, and as later clarified in \cite{CK} unit-regularity for rings is equivalent to a strengthened form of cleanness.  These facts cannot be significantly weakened, since there are regular rings which are not clean.  Nor can these statements be significantly strengthened, because the example from the previous section demonstrates that unit-regular rings need not be strongly clean rings.

The purpose of this section is to clarify the connection between unit-regular elements and clean elements, by showing that a ``doubly unit-regular'' condition is equivalent to a strengthened form of cleanness.  This result should come as something of a surprise, since it was shown in \cite{KL} that unit-regular elements are not necessarily clean.  Before we state our main result, we first generalize some results in the literature by characterizing when matrices with a column of zeros are clean.

To begin we recall an easy alternate characterization of when $a\in R$ is clean.

\begin{lemma}[{\cite[Proposition 2]{Zhang}}]\label{Lemma:CleanAlt}
Let $a$ be an element of a ring $R$.  The following are equivalent:
\begin{enumerate}[label={\rm (\arabic*)}]
\item
There exist $e\in \operatorname{idem}(R)$ and $u\in U(R)$ such that $a=e+u$.

\item
There exist $g\in \operatorname{idem}(R)$ and $v\in U(R)$ such that $g=gva$ and $1-g=-(1-g)v(1-a)$.
\end{enumerate}
Moreover, there is a natural bijection between the two conditions.  Given $u,e$, we can take $v=u^{-1}$ and $g=1-u^{-1}eu$; conversely given $v,g$ we can take $u=v^{-1}$ and $e=1-v^{-1}gv$.
\end{lemma}

Let $R$ be a ring with an idempotent $e$ and $a\in Re$. In the Peirce decomposition corresponding to $e$ and $f:=1-e$, we can write $a$ as
\begin{equation}\label{matfirstcol}
a=\left(\begin{array}{cc}\alpha&0\\\tau&0\end{array}\right),
\end{equation}
where $\alpha:=ea\in eRe$ and $\tau:=fa\in fRe$.  The following lemma characterizes when such an element $a$ is clean in $R$, in terms of the properties of the ``corner'' elements $\alpha$ and $\tau$.  The lemma generalizes \cite[Proposition 2.2]{SterCorner} which deals with matrices of the form (\ref{matfirstcol}) with the additional property $\tau=0$, and also \cite[Theorem 3.2]{KL} which deals with matrices in $\M_2(K)$ with a zero row (or column) in the case when the base ring $K$ is \emph{commutative}.

\begin{lemma}\label{lemfirstcol}
Let $R$ be a ring, $e\in \operatorname{idem}(R)$, $f:=1-e$, and $a\in Re$.  Put $\alpha:=eae$ and $\tau:=fae$ as above.
The following are equivalent:
\begin{enumerate}[label={\rm (\arabic*)}]
\item \label{colclean}
The element $a$ is clean in $R$.
\item \label{doubleeq}
There exist $\varepsilon\in\operatorname{idem}(eRe)$, $\mu\in U(eRe)$, $\beta\in eRf$, and $\gamma\in fRe$ such that
\[
\varepsilon=\varepsilon\mu\alpha+\varepsilon\beta(\tau+\gamma\alpha) \qquad\textnormal{\textit{and}}\qquad e-\varepsilon=-(e-\varepsilon)\mu(e-\alpha).
\]
\item \label{singleeq}
There exist $\varepsilon\in\operatorname{idem}(eRe)$, $\mu\in U(eRe)$, $\beta\in eRf$, and $\gamma\in fRe$ such that
\[
\alpha=\varepsilon+\mu+(e-\varepsilon)\beta(\tau+\gamma\alpha).
\]
\end{enumerate}
\end{lemma}
\begin{proof}
\ref{colclean} $\Rightarrow$ \ref{doubleeq}: We imitate the proof of \cite[Proposition 2.2]{SterCorner}. Assume that $a$ is clean. Write $g=gva\quad\textnormal{and}\quad 1-g=-(1-g)v(1-a)$ where $g\in\operatorname{idem}(R)$ and $v\in U(R)$. Since $g\in Ra\subseteq Re$, it must be the case that $g$ has a zero second column (in the $e$-$f$-Peirce decomposition), so that
\[
g=\left(\begin{array}{cc}\varepsilon&0\\\chi&0\end{array}\right)
\]
for some $\varepsilon\in \operatorname{idem}(eRe)$ and $\chi\in fR\varepsilon$. Write $\zeta:=e-\varepsilon$ and set
\[
v=\begin{pmatrix}\pi & \beta\\ \sigma &\delta\end{pmatrix}.
\]
By matrix expansion of the equation $g=gva$ we get the following two equalities:
\begin{eqnarray}
\label{Eq:first} \varepsilon & = &
\varepsilon\pi\alpha+\varepsilon\beta\tau,\\
\label{Eq:second} \chi & = & \chi\pi\alpha+\chi\beta\tau.
\end{eqnarray}
(Note that (\ref{Eq:second}) also follows from (\ref{Eq:first}) since $\chi=\chi\varepsilon$.) Furthermore, from $1-g=-(1-g)v(1-a)$ we get the following four equations:
\begin{eqnarray}
\label{Eq:third} \zeta & = & -\zeta\pi(e-\alpha)+\zeta\beta\tau,\\
\label{Eq:fourth} 0 & = & \zeta\beta,\\
\label{Eq:fifth} -\chi & = &
-(\sigma-\chi\pi)(e-\alpha)+(\delta-\chi\beta)\tau,\\
\label{Eq:sixth} f & = & -\delta+\chi\beta.
\end{eqnarray}
The invertibility of $v$, in conjunction with (\ref{Eq:sixth}) yields that
\[
\left(\begin{array}{cc}e&\beta\\0&f\end{array}\right)
\left(\begin{array}{cc}e&0\\-\chi&f\end{array}\right)v=
\left(\begin{array}{cc}\pi+\beta(\sigma-\chi\pi)&0\\ \sigma-\chi\pi&
-f\end{array}\right)
\]
is invertible. As the lower-right corner is a unit (in the corner ring $fRf$) and the entire matrix is a lower-triangular invertible matrix, we see that $\mu:=\pi+\beta(\sigma-\chi\pi)$ is invertible in $eRe$.

From (\ref{Eq:fourth}) we have $\zeta\pi=\zeta\mu$, so that (\ref{Eq:third}) becomes $\zeta=-\zeta\mu(e-\alpha)$, which is the second of the two equations we need. Moreover, (\ref{Eq:first}) yields $\varepsilon=\varepsilon(\mu-\beta(\sigma-\chi\pi))\alpha+\varepsilon\beta\tau= \varepsilon\mu\alpha+\varepsilon\beta(\tau-(\sigma-\chi\pi)\alpha)$. Taking $\gamma:=-\sigma+\chi\pi$ gives the other needed equation. (It may be interesting to note that we never needed to use (\ref{Eq:fifth}).)

\ref{doubleeq} $\Rightarrow$ \ref{colclean}:
Suppose that $\varepsilon=\varepsilon\mu\alpha+\varepsilon\beta(\tau+\gamma\alpha)$ and $\zeta=-\zeta\mu(e-\alpha)$ for some $\varepsilon=e-\zeta\in\operatorname{idem}(eRe)$, $\mu\in U(eRe)$, $\beta\in eRf$ and $\gamma\in fRe$. Write $\tau':=\tau+\gamma\alpha-\gamma\in fRe$, and let
\[
g:=\left(\begin{array}{cc}\varepsilon&0\\\tau'\varepsilon&0\end{array}\right)
\]
and
\[
v:=\left(\begin{array}{cc}e&0\\\tau'\varepsilon&f\end{array}\right)
\left(\begin{array}{cc}e&-\varepsilon\beta\\0&f\end{array}\right)
\left(\begin{array}{cc}(e+\varepsilon\beta\tau'\zeta)\mu&0\\\tau'\zeta\mu-\gamma&-f\end{array}\right).
\]
Clearly, $g\in\operatorname{idem}(R)$ and $v\in U(R)$, and a straightforward verification shows that $g=gva$ and $1-g=-(1-g)v(1-a)$.  Thus $a$ is clean in $R$ by Lemma \ref{Lemma:CleanAlt}.

\ref{doubleeq} $\Rightarrow$ \ref{singleeq}: The proof of this implication and its converse are similar to the proof of Lemma \ref{Lemma:CleanAlt}, but we include the details for completeness.  Given the two equations of \ref{doubleeq}, adding them together we get
\[
e=\varepsilon\mu\alpha+\varepsilon\beta(\tau+\gamma\alpha)-(e-\varepsilon)\mu(e-\alpha)= \mu\alpha+\varepsilon\beta(\tau+\gamma\alpha)-(e-\varepsilon)\mu,
\]
so that
\[
\alpha=\mu^{-1}(e-\varepsilon)\mu+\mu^{-1}-\mu^{-1}\varepsilon\beta(\tau+\gamma\alpha),
\]
which gives the desired equation taking $\varepsilon':=\mu^{-1}(e-\varepsilon)\mu$, $\mu':=\mu^{-1}$, $\beta':=-\mu^{-1}\beta$, and $\gamma':=\gamma$.

\ref{singleeq} $\Rightarrow$ \ref{doubleeq}: Starting with the equation given in \ref{singleeq}, multiplying on the left by $\mu^{-1}\varepsilon$ gives $\mu^{-1}\varepsilon\alpha=\mu^{-1}\varepsilon+\mu^{-1}\varepsilon\mu$.  Hence
\begin{equation}\label{Eq:lastone}
\mu^{-1}\varepsilon\mu=-(\mu^{-1}\varepsilon\mu)\mu^{-1}(e-\alpha).
\end{equation}
Taking $\varepsilon':=\mu^{-1}(e-\varepsilon)\mu$ and $\mu':=\mu^{-1}$, then (\ref{Eq:lastone}) is exactly the second equation of \ref{doubleeq}. Similarly, multiplying \ref{singleeq} on the left by $\mu^{-1}(e-\varepsilon)$ gives the first equation of \ref{doubleeq}.
\end{proof}

\begin{remark}
(1) If $\tau=0$, then Lemma \ref{lemfirstcol} says that $a=\alpha\in eRe$ is clean in $R$ if and only if
\begin{equation}\label{Eq:equivclean}
\alpha=\varepsilon+\mu +(e-\varepsilon)\beta\gamma\alpha
\end{equation}
for some $\varepsilon\in\operatorname{idem}(eRe)$, $\mu\in U(eRe)$, $\beta\in eRf$, and $\gamma\in fRe$. Thus $a$ is \textit{weakly clean} in $eRe$, following the terminology of \cite[Definition 2.3]{SterCorner}.

Conversely, if $a\in eRe$ is weakly clean in $eRe$ and
\begin{equation}\label{Eq:ExtraAssumption}
e=exfye \text{ for some } x,y\in R,
\end{equation}
then $a$ satisfies (\ref{Eq:equivclean}) and hence is clean in $R$.  However, without assuming any extra condition such as (\ref{Eq:ExtraAssumption}), weakly clean elements of corner rings need not be clean in the entire ring.  This is easy to see by taking $e=1$ and $R$ to be any weakly clean ring which is not clean, such as in \cite[Example 3.1]{SterCorner}.

(2) Lemma \ref{lemfirstcol} provides an easy way to see that regular elements in rings with stable range one are always clean, which was originally proved in \cite[Theorem 3.3]{WCKL}.

To see this, let $R$ have stable range one and write $a=ara\in R$ with $r\in R$. Setting $e:=ra$ we have $a\in Re$, so that $a$ decomposes as in (\ref{matfirstcol}).  Since $e=ere\alpha+erf\tau$ and $eRe$ has stable range one, we have $e=\mu\alpha+\omega erf\tau$ for some $\mu\in U(eRe)$ and $\omega\in eRe$. Thus, $\varepsilon:=e$, $\mu$, $\beta:=\omega erf$, and $\gamma:=0$ satisfy Lemma \ref{lemfirstcol}\ref{doubleeq}, and hence $a$ is clean in $R$.
\end{remark}

We are now prepared to present the main theorem of this section, which demonstrates that certain unit-regular elements, which we call ``doubly unit-regular,'' possess an extended version of the clean property.

As mentioned previously, not all unit-regular elements are clean.  To get around this problem, we assume two instances of unit-regularity; once for the original element, and once for a corner of the element.  However, as it turns out, unit-regularity is only needed for the corner and mere regularity for the original element.

\begin{thm}\label{correspondence}
Let $a$ be an element of a ring $R$. The following are equivalent:
\begin{enumerate}[label={\rm (\arabic*)}]
\item \label{equiv1}
There exists $u\in U(R)$ with $aua=a$, and writing $e:=ua\in\operatorname{idem}(R)$, the element $eae$ is unit-regular in $eRe$.
\item \label{equiv2}
The exists $r\in R$ with $ara=a$, and writing $e:=ra\in\operatorname{idem}(R)$, the element $eae$ is unit-regular in $eRe$.
\item\label{equiv3}
There exist $g\in\operatorname{idem}(R)$ and $v\in U(R)$ such that putting $h:=1-g$ we have $g=gva$, $h=-hv(1-a)$, $hvh=-h$, and $gvhvgvh=-gvh$.
\item \label{equiv4}
There exist $e\in\operatorname{idem}(R)$ and $u\in U(R)$ such that $a=e+u$, $aR\cap eR=(0)$, and $a^2R\cap aeR=(0)$.
\item \label{equiv5}
There exist $e\in\operatorname{idem}(R)$ and $u\in U(R)$ such that $a=e+u$, $aR\cap eR=(0)$, and $a^2R\cap aeaR=(0)$.
\item \label{equiv6}
There exist $e\in \operatorname{idem}(R)$ and $u\in U(R)$ such that $a=e+u$ with $au^{-1}a=a$ and $a^2u^{-2}a^2=a^2$.
\end{enumerate}
\end{thm}
\begin{proof}
\ref{equiv1} $\Rightarrow$ \ref{equiv2} is a tautology.

\ref{equiv2} $\Rightarrow$ \ref{equiv3}: Put $f:=1-e$. Since $a\in Re$, we can write $a=\left(\begin{smallmatrix}\alpha&0\\\tau&0\end{smallmatrix}\right)$ in the $e$-$f$-Peirce decomposition. As $\alpha$ is unit-regular by hypothesis, we can find $\mu\in U(eRe)$ such that $\alpha\mu\alpha=\alpha$. Now
\[
e=e^3=erae=ereae+erfae=ere\alpha+erf\tau
\]
gives $erf\tau(e-\mu\alpha)=(e-ere\alpha)(e-\mu\alpha)=e-\mu\alpha$, so that
\begin{equation}\label{Eq:newer}
e=\mu\alpha+(e-\mu\alpha)erf\tau(e-\mu\alpha)=\mu\alpha+(e-\mu\alpha)erf(\tau-\tau\mu\alpha).
\end{equation}
Hence $\varepsilon:=e$, $\mu$, $\beta:=(e-\mu\alpha)erf$, and $\gamma:=-\tau\mu$ satisfy the conditions of Lemma \ref{lemfirstcol}\ref{doubleeq}. Accordingly, $a$ is clean in $R$, i.e.\ $g=gva$ and $h=-hv(1-a)$, with $g=1-h\in\operatorname{idem}(R)$ and $v\in U(R)$. Moreover, the proof of the implication \ref{doubleeq} $\Rightarrow$ \ref{colclean} in Lemma \ref{lemfirstcol} tells us that we can take
\[
g:=\left(\begin{array}{cc}e&0\\\tau'&0\end{array}\right) \quad\textnormal{and}\quad
v:=\left(\begin{array}{cc}e&0\\\tau'&f\end{array}\right)
\left(\begin{array}{cc}e&-\beta\\0&f\end{array}\right)
\left(\begin{array}{cc}\mu&0\\-\gamma&-f\end{array}\right)
\]
where $\tau':=\tau+\gamma\alpha-\gamma$ (since $\varepsilon=e$ and hence $\zeta=0$). It is a matter of a routine verification that $g$, $h$, and $v$ satisfy $hvh=-h$, so it only remains to see that $gvhvgvh=-gvh$. By direct computation we get
\[
gvh=
\left(\begin{array}{cc}-\beta\tau'&\beta\\-\tau'\beta\tau'&\tau'\beta\end{array}\right)
\quad\textnormal{and}\quad
gvhvgvh=
\left(\begin{array}{cc}\beta(\gamma+\tau')\beta\tau'&-\beta(\gamma+\tau')\beta\\
\tau'\beta(\gamma+\tau')\beta\tau'&-\tau'\beta(\gamma+\tau')\beta\end{array}\right),
\]
so we only need to see that $\beta(\gamma+\tau')\beta=\beta$. This follows from $\beta=(e-\mu\alpha)\beta$ and
\[
\beta(\gamma+\tau')=\beta(\tau+\gamma\alpha)=(e-\mu\alpha)erf(\tau-\tau\mu\alpha) =e-\mu\alpha,
\]
where the last equality follows from (\ref{Eq:newer}).

\ref{equiv3} $\Rightarrow$ \ref{equiv4}: Assuming \ref{equiv3}, Lemma \ref{Lemma:CleanAlt} gives us an idempotent $e:=v^{-1}hv$ and a unit $u:=v^{-1}$ such that $a=e+u$. From $hvh=-h$ we have $-ve=-hv\in\operatorname{idem}(R)$. The orthogonal complement of this idempotent is precisely $1+ve=1+v(a-v^{-1})=va$. As orthogonal idempotents have disjoint images, it follows that $vaR\cap(-ve)R=(0)$, which is nothing but $aR\cap eR=(0)$.

To prove the remaining equality $a^2R\cap aeR=(0)$, first apply $hv=hvg-h$ and $vh=gvh-h$ to get
\[
v^2ae=v(va)e=v(1+hv)e=(1+vh)ve=(1+vh)hv=(gvh+g)(hvg-h)=gvhvg-gvh
\]
and
\[
v^2a^2=v(va)a=v(1+hv)a=(1+vh)va=(1+vh)(1+hv)=(gvh+g)(hvg+g)=gvhvg+g.
\]
Using $gvhvgvh=-gvh$, it is now a routine verification to see that $-v^2ae$ and $v^2a^2$ are orthogonal idempotents. Therefore $v^2a^2R\cap(-v^2ae)R=(0)$, which readily gives the needed equality.

\ref{equiv4} $\Rightarrow$ \ref{equiv5} is trivial since $aeaR\subseteq aeR$.

\ref{equiv5} $\Rightarrow$ \ref{equiv6}: Assume that \ref{equiv5} holds. Denote $f:=1-e$, then
\[
f(au^{-1}a-a)=fa(u^{-1}a-1)=f(e+u)(u^{-1}a-1)=fu(u^{-1}a-1)=f(a-u)=fe=0.
\]
Hence $au^{-1}a-a\in aR\cap eR=(0)$ and therefore $au^{-1}a=a$. Moreover, from $eu^{-1}a=(a-u)u^{-1}a=0$ (using $aR\cap eR=(0)$) we have $afau^{-2}a=af(e+u)u^{-2}a=afu^{-1}a=au^{-1}a=a$, which gives
\[
a^2u^{-2}a^2-a^2=aeau^{-2}a^2+afau^{-2}a^2-a^2=aeau^{-2}a^2.
\]
Thus, $a^2u^{-2}a^2-a^2\in a^2R\cap aeaR=(0)$, which yields $a^2u^{-2}a^2=a^2$.

\ref{equiv6} $\Rightarrow$ \ref{equiv1}: Write $e':=u^{-1}a$. We will prove that $e'ae'=e'a$ is unit-regular in $e'Re'$. Set $f:=1-e$. From
\[
fue'=fuu^{-1}a=fa=f(e+u)=fu
\]
and
\[
e'fu=u^{-1}afu=u^{-1}(e+u)fu=fu
\]
we have $fu\in e'Re'$. This, together with $ee'=eu^{-1}a=(a-u)u^{-1}a=0$ and $au^{-1}e=au^{-1}(a-u)=0$, gives
\[
fue'u^{-1}e'=fuu^{-1}e'=fe'=e'
\]
and
\[
e'u^{-1}e'fu=e'u^{-1}fu=u^{-1}au^{-1}fu=u^{-1}au^{-1}u=u^{-1}a=e'.
\]
Hence $e'u^{-1}e'$ is a unit in $e'Re'$, with inverse $fu$. We also have
\[
e'ae'u^{-1}e'a=e'au^{-1}e'a=u^{-1}a^2u^{-2}a^2=u^{-1}a^2=e'a
\]
which proves that indeed $e'a\in\operatorname{ureg}(e'Re')$. This completes the proof of the theorem.
\end{proof}

Condition \ref{equiv6} of Theorem \ref{correspondence} is visually left-right symmetric, and so we could replace any of the other conditions by their left-right analogue.  Further, assume $a\in R$ is unit-regular.  Let $r,s$ be any two inner inverses, so $ara=asa=a$.  Putting $e=ra$ and $e'=sa$ we claim that $eae$ is unit-regular in $eRe$ if and only if $e'ae'$ is unit-regular in $e'Re'$, and hence conditions \ref{equiv1} and \ref{equiv2} do not depend on which (unit) inner inverse one chooses.  This can be shown directly, but is also a consequence of the following:

\begin{prop}\label{Prop:Item7}
If $R$ is the endomorphism ring of some left $k$-module $\!_kM$, then the conditions of {\rm Theorem \ref{correspondence}} are also equivalent to:
\begin{enumerate}[label={\rm (\arabic*)}, start=7]
\item \label{equiv7}
The element $a\in \operatorname{End}(M)$ is regular in $\operatorname{End}(M)$, and $a|_{Ma}\in \operatorname{End}(Ma)$ is unit-regular in $\operatorname{End}(Ma)$.
\end{enumerate}
\end{prop}
\begin{proof}
\ref{equiv2} $\Leftrightarrow$ \ref{equiv7}: Assume $a\in R$ is regular, and let $r\in R$ be any inner inverse.  Putting $e:=ra\in \operatorname{idem}(R)$, we then have $Ma=Me$.  There is a natural identification $\operatorname{End}(Ma)=\operatorname{End}(Me)\isom eRe$.  Under this identification, $a|_{Ma}$ corresponds to $eae$.  Thus $eae$ is unit-regular if and only if $a|_{Ma}$ is as well.
\end{proof}

\begin{remark}
(1) As Theorem \ref{correspondence} is an element-wise statement, in principle one should be able to give precise formulas for some clean decomposition of $a$, using the following two conditions: (A) $ara=a$ for some $r\in R$ and (B) $w\in U(raRra)$ is an inner inverse for the element $(ra)a(ra)$ in the corner ring $raRra$.  Indeed, put $t=w^{-1}\in U(raRra)$.  With some work we obtain the relations
\[
\begin{array}{llll}
(1)\ ara=a, & (2)\ a^2wa=a^2, & (3)\ wawa=wa, & (4)\ tawa=ta\\
(5)\ wra=w, & (6)\ raw=w, & (7)\ tra=t, & (8)\ rat=t,\\
(9)\ wt=ra, & (10)\ tw=ra.
\end{array}
\]
Putting
\[
e := 1-ra+tr+ar^2a-artr-awaw-ra^2r+ra^2w+ar^2a^2r-ar^2a^2w-awar^2a+awar^2a^2w
\]
then a direct computation (which we performed using a computer and only the ten relations above) shows $e^2=e$ and $u:=a-e$ is a unit satisfying condition \ref{equiv6} of Theorem \ref{correspondence}, whose inverse $v$ has 53 monomials in its support when written in the letters $a,r,t,w$!  (It is possible that a different choice for $e$ may lead to a slightly simpler expression for $v$.  However, adjoining an inverse for $r$ does not simplify any of the formulas given here.)

(2) As one may expect, the statements \ref{equiv4} and \ref{equiv5} of Theorem \ref{correspondence} are \emph{not} equivalent in the sense that a fixed idempotent $e$ and a unit $u$ would satisfy the conditions of \ref{equiv4} if and only if they would satisfy the conditions of \ref{equiv5}. For example, taking $a=\left(\begin{smallmatrix}0&0\\1&1\end{smallmatrix}\right)$, $e=\left(\begin{smallmatrix}1&0\\0&0\end{smallmatrix}\right)$ and $u=\left(\begin{smallmatrix}-1&0\\1&1\end{smallmatrix}\right)$ in the $2\times 2$ matrix ring over a field, one easily checks that $a,e,u$ satisfy \ref{equiv5} but $a^2R\cap aeR\ne(0)$.
\end{remark}

A well-known result of Ara says that strongly $\pi$-regular rings have stable range 1.  A key step in the proof is the observation that any regular, nilpotent element of an exchange ring is a unit-regular element (by \cite[Theorem 2]{Ara}, but also see \cite{Khurana}).  One might ask: Is the assumption that $R$ is an exchange ring necessary?  If $a\in \operatorname{reg}(R)$ and $a^2=0$, then Proposition \ref{Prop:Item7} tells us $a$ is doubly unit-regular, and thus \emph{both} unit-regular and clean.

However, to end this section we construct an example of a ring $S$ and a regular element $a\in S$ such that $a^3=0$, but $a$ is not unit-regular in $S$.  Our construction is based on the ring $R=F\langle x,y\ :\ x^2=0\rangle$.  The following lemma lists some properties of this ring.

\begin{lemma}\label{xylemma}
Let $F$ be a field and $R=F\langle x,y\ :\ x^2=0\rangle$. The following hold:
\begin{enumerate}[label={\rm (\arabic*)}]
\item \label{ZD1}
If $ab=0$ for some nonzero $a,b\in R$ then $a\in Rx$ and $b\in xR$. In particular, the set of nilpotent elements of $R$ is precisely $Rx\cap xR=Fx+xRx$.
\item \label{ZD2}
The idempotents of $R$ are trivial, so $\operatorname{idem}(R)=\{0,1\}$.
\item \label{ZD3}
The ring $R$ is directly finite.
\item \label{ZD4}
Units in $R$ are exactly the elements of the form $\mu+a$ where $\mu\in F\setminus\{0\}$ and $a\in Rx\cap xR$. In particular, $U(R)+Fx\subseteq U(R)$, and $1-yx\notin U(R)$.
\end{enumerate}
\end{lemma}
\begin{proof}
\ref{ZD1} is found in \cite[Example 9.3]{CamilloNielsen}.

\ref{ZD2}:  If $e$ is a nontrivial idempotent in $R$ then $e$ is both a left and a right zero divisor, so that $e\in Rx\cap xR$ by \ref{ZD1}.  Hence $e$ is a nilpotent and thus $e=0$, a contradiction.  (Alternatively, this follows by an easy minimal degree argument.)

\ref{ZD3} follows from \ref{ZD2} since $ab=1$ always implies that $ba$ is a nonzero idempotent (when $1\neq 0$).

\ref{ZD4}:  Clearly $(F\setminus\{0\})+(Rx\cap xR)\subseteq U(R)$, so it suffices to prove the other inclusion.  Let $u\in U(R)$. We may write $u=\mu+u_1+u_2+u_3+u_4$ with $\mu\in F$, $u_1\in Rx\cap xR$, $u_2\in Rx\cap yR$, $u_3\in Ry\cap xR$ and $u_4\in Ry\cap yR$.  Clearly, $\mu\ne 0$. We need to prove that $u_2=u_3=u_4=0$.

Let $v:=u^{-1}=\nu+v_1+v_2+v_3+v_4$, with $\nu\in F$, $v_1\in Rx\cap xR$, $v_2\in Rx\cap yR$, $v_3\in Ry\cap xR$, and $v_4\in Ry\cap yR$.  We have $xu\cdot vx=0$ and $xu,vx\ne 0$, so that \ref{ZD1} yields $xu\in Rx$. This gives $xu_4\in Rx$, so that $u_4=0$. Similarly, $v_4=0$.

Suppose that $u_2\neq 0$ and $v_2\neq 0$. Taking any monomial $p$ in $u_2$ of the largest degree, and any monomial $q$ in $v_2$ of the largest degree, we see that the monomial $pq$ cannot cancel with any other monomial in the product $uv$, so that $uv\neq 1$, which is a contradiction.  Thus $u_2\neq 0$ forces $v_2=0$. Similarly, $u_3\neq 0$ forces $v_3=0$. Therefore, if both $u_2,u_3\neq 0$ then $v=\nu+v_1$, which gives $u=v^{-1}=\nu^{-1}-\nu^{-2}v_1$, a contradiction. Hence $u_2=0$ or $u_3=0$; we may assume by symmetry that $u_3=0$.

Finally, suppose that $u_2\neq 0$, so that $v_2=0$. Then $x=xvu=x(\nu+v_1+v_3)(\mu+u_1+u_2)=\nu x(\mu+u_1+u_2)=\mu\nu x+\nu xu_2$. Hence $xu_2\in Fx$, which is again a contradiction. Thus $u_2=0$, which completes the proof.
\end{proof}

\begin{example}\label{Ex:RegNilp}
There exists a ring $S$ and an element $a\in \operatorname{reg}(S)$ with $a^3=0$, but $a\notin \operatorname{ureg}(S)$.
\end{example}
\begin{proof}
Let $F$ be a field, and set $R=F\langle x,y\ :\ x^2=0\rangle$.  Let $I=R(1-yx)$ which is a left ideal in $R$.  We let $S$ be the unital subring of $\M_2(R)$ given by
\[
S=\begin{pmatrix}R & I\\ R & F+I\end{pmatrix}.
\]
The element $A=\left(\begin{smallmatrix}x & 0\\ 1 & 0\end{smallmatrix}\right)$ is regular with inner inverse $\left(\begin{smallmatrix}y & 1-yx\\ 0 & 0\end{smallmatrix}\right)$.  Further $A^3=0$.

Assume, by way of contradiction, that $A$ is unit-regular in $S$, so $A=AUA$ for some unit $U$ of $S$.  As $E=UA$ is an idempotent with a second column of zeros, it is of the form
\[
E=\begin{pmatrix}e & 0\\ te & 0\end{pmatrix}
\]
for some $t\in R$, and some nonzero idempotent $e^2=e\in R$.  The ring $R$ has only trivial idempotents so $e=1$.  Now $V=U^{-1}\left(\begin{smallmatrix}1 & 0\\ t & 1\end{smallmatrix}\right)$ is an invertible matrix in $S$ satisfying
\[
V\begin{pmatrix}1 & 0\\ 0 & 0\end{pmatrix}=U^{-1}\begin{pmatrix}1 & 0\\ t & 1\end{pmatrix}\begin{pmatrix}1 & 0\\ 0 & 0\end{pmatrix}=U^{-1}E=A,
\]
hence the first columns of $V$ and $A$ coincide, so we may write $V=\left(\begin{smallmatrix}x & a\\ 1 & b\end{smallmatrix}\right)$ for some $a\in I$ and $b\in F+I$.

Now, since $V$ is invertible in the larger ring $\M_2(R)$, we see that
\[
V\begin{pmatrix}-b & 1\\ 1 & 0\end{pmatrix}=\begin{pmatrix}a-xb & x\\ 0 & 1\end{pmatrix}
\]
is also invertible, and hence $c:=a-xb\in U(R)$.  On the other hand $xb\in I+Fx$ so that $c\in I+Fx$.  Writing $c=\lambda x+c'(1-yx)$ with $\lambda\in F$ and $c'\in R$, by Lemma \ref{xylemma}\ref{ZD4} we have
\[
c'(1-yx)=c-\lambda x\in U(R)
\]
so that $1-yx\in U(R)$, yielding the needed contradiction.
\end{proof}

\begin{remark}
Another ring containing a nilpotent regular element which is not unit-regular was quite recently found by Ara and O'Meara, and published on the arXiv \cite{AM}.  They work directly with the generic ring $F\langle a,x \ :\ axa=a,xax=x,a^3=0\rangle$.  Their work was published after we had discovered the argument above, and as we believe both examples and methods are interesting in their own right, we continue to include our example here.
\end{remark}

It may be interesting to note that there is a special case which has a positive solution:

\begin{prop}
If $a,b\in R$ are nilpotent and $aba=a$, then $a\in {\rm ureg}(R)$.
\end{prop}
\begin{proof}
Let $u:=b+(1+a)^{-1}(1-ab)$.  Clearly, $ua=ba$ and so $aua=aba=a$.  We also compute
\[
u=(1+a)^{-1}(1+a)b + (1+a)^{-1}(1-ab) = (1+a)^{-1}(b+ab+1-ab)=(1+a)^{-1}(1+b)\in U(R).\qedhere
\]
\end{proof}

This proposition leaves open the possibility that a nilpotent element with a strongly $\pi$-regular inner inverse is unit-regular; but we were unable to solve that problem.

\section{Powers of regular elements}\label{Section:Powers}

Condition \ref{equiv6} of Theorem \ref{correspondence} is interesting not only because it tells us that doubly unit-regular elements are clean, but also due to its implications regarding the inner inverses of powers of unit-regular elements.  This section is dedicated to revealing even more structure in this regard, with some surprising consequences.  Our first theorem tells us that when powers of an element are von Neumann regular, then those powers \emph{always} have inner inverses which are also powers.  Thus, only an extremely minimal hypothesis is needed to guarantee this ``power inner inverse'' condition.

\begin{thm}\label{Thm:PowerReg}
If $a,a^2,\ldots,a^n\in \operatorname{reg}(R)$ for some $n\geq 1$, then there exists $w\in R$ such that
\begin{equation}\label{powers}
a^iw^ja^j=w^{j-i}a^j\qquad\textnormal{and}\qquad a^jw^ja^i=a^jw^{j-i}, \qquad \text{ for all $1\leq i\leq j\leq n$}.
\end{equation}
In particular, $a^iw^ia^i=a^i$ for all $1\leq i\leq n$.  Moreover, if $a\in \operatorname{ureg}(R)$, then one can take $w\in U(R)$.
\end{thm}
\begin{proof}
We prove the theorem by induction on $n\geq 1$. The case $n=1$ being trivial, let $n\geq 2$ and assume that the theorem holds for integers less than $n$. Assume $a,a^2,\ldots,a^n \in \operatorname{reg}(R)$ and fix $r\in R$ such that $a=ara$. In the case when $a$ is unit-regular, we also assume that $r\in U(R)$.

Write $e:=ra\in \operatorname{idem}(R)$, $f:=1-e$, and $\alpha:=ea\in eRe$.  We have $\alpha^k=ea^k=ra^{k+1}$ for every $k$, so that the regularity of $a^2,a^3,\ldots,a^n$, together with $Ra=Rra$, yields
\[
\alpha^k=ra^{k+1}\in ra^{k+1}Ra^{k+1}=ra^{k+1}Rra^{k+1}=\alpha^kR\alpha^k
\]
for every $1\leq k\leq n-1$. Therefore, by the inductive hypothesis there exists $v\in R$ such that $\alpha^iv^j\alpha^j=v^{j-i}\alpha^j$ for all $1\le i\le j\le n-1$. (We could also assume that $\alpha^jv^j\alpha^i=\alpha^jv^{j-i}$, but we will not need this fact.)

Similarly, let $e':=ar\in \operatorname{idem}(R)$, $f':=1-e'$, and $\alpha':=ae'\in e'Re'$. As above, we see that $\alpha',\alpha'^2,\ldots,\alpha'^{n-1}$ are regular, so that we can find $v'\in R$ satisfying $\alpha'^jv'^j\alpha'^i=\alpha'^jv'^{j-i}$ for all $1\le i\le j\le n-1$.

Now define
\[
w:=(1+fave)r(1+e'v'af').
\]
We will prove that $w$ satisfies the desired properties. First, it is clear that if $a$ is unit-regular then $w$ is a unit since it is a product of three units. It remains to see that \eqref{powers} holds for all $1\le i\le j\le n$. We only need to verify the first of the two equations, since the other one will follow by symmetry (as $w$ is defined in a left-right symmetric way).

Initially, let us prove that
\begin{equation}
\label{onej}
aw^ja^j=w^{j-1}a^j\qquad\textnormal{for all}\quad 1\leq j\leq n.
\end{equation}
If $j=1$ this is clear because $af=f'a=0$ implies the equality
\[
awa=a(1+fave)r(1+e'v'af')a=ara=a.
\]
Proceeding inductively, let $j\geq 2$ and suppose that the statement holds for integers less than $j$. From $fav\alpha=av\alpha-\alpha v\alpha=av\alpha-\alpha$ we get
\begin{equation}
\label{hoplaa}
wa^2=(1+fave)r(1+e'v'af')a^2=(1+fave)ra^2=(1+fave)\alpha=\alpha+fav\alpha=av\alpha.
\end{equation}
This, together with $v\alpha^{j-1}=\alpha^{j-2}v^{j-1}\alpha^{j-1}$, gives
\begin{equation}
\label{importa}
wa^j=wa^2a^{j-2}=av\alpha a^{j-2}=av\alpha^{j-1}=a\alpha^{j-2}v^{j-1}\alpha^{j-1}=a^{j-1}v^{j-1}\alpha^{j-1}.
\end{equation}
Now, applying \eqref{importa} twice and considering $aw^{j-1}a^{j-1}=w^{j-2}a^{j-1}$, we get
\[
aw^ja^j=aw^{j-1}wa^j=aw^{j-1}a^{j-1}v^{j-1}\alpha^{j-1}=w^{j-2}a^{j-1}v^{j-1}\alpha^{j-1}=w^{j-2}wa^j=w^{j-1}a^j.
\]
This completes the proof of \eqref{onej}.

Now let us prove that $a^iw^ja^j=w^{j-i}a^j$ for all $1\leq i\leq j\leq n$.  If $i=1$ this is just \eqref{onej}. Proceeding inductively, let $i\geq 2$ and suppose that the statement holds for all $1\leq i'<j\leq n$ where $i'<i$. Then, applying $aw^ja^j=w^{j-1}a^j$ and $a^{i-1}w^{j-1}a^{j-1}=w^{j-i}a^{j-1}$, we quickly obtain
\[
a^iw^ja^j=a^{i-1}aw^ja^j=a^{i-1}w^{j-1}a^j=w^{j-i}a^{j-1}a=w^{j-i}a^j,
\]
which is what was to be proved. This completes the proof of the theorem.
\end{proof}

\begin{remark}
(1) Given $ax_1a=a$ and $a^2 x_2 a^2=a^2$, the element
\begin{eqnarray*}
w & := & x_1 +ax_2ax_1+x_1ax_2a-x_1a^2x_2ax_1-x_1ax_2a^2x_1 +\\ & & ax_2ax_2a-ax_2ax_2a^2x_1-x_1a^2x_2ax_2a+x_1a^2x_2ax_2a^2x_1
\end{eqnarray*}
satisfies the conditions of the theorem.  We leave the quick check to the interested reader.

(2)  If $a$ is an element of a unit-regular ring $R$ and $r$ is some (invertible) element with $ara=a$, then $a^kr^ka^k=a^k$ can easily fail for every $k\geq 2$. For example, let $R=\M_2(F(x))$, with $F$ a field and $F(x)$ the field of rational functions over $F$.  Put $a=\left(\begin{smallmatrix}1&0\\0&0\end{smallmatrix}\right)$ and $r=\left(\begin{smallmatrix}1&1\\x&x^2\end{smallmatrix}\right)$.  An easy inductive argument shows that
\[
r^k=\begin{pmatrix}s_k(x) & t_k(x)\\ u_k(x) & v_k(x)\end{pmatrix}
\]
with $\deg(s_k)=2k-3$, $\deg(t_k)=2k-2$, $\deg(u_k)=2k-1$, and $\deg(v_k)=2k$, when $k\geq 2$.  Thus, $a^kr^ka^k\neq a^k$ when $k\geq 2$.  A quick check shows, however, that $ara=a$.
\end{remark}

The next corollary is also a consequence of \cite[Lemma 3.2]{HO}.

\begin{cor}
If $a$ is a unit-regular element in a regular ring $R$, then $a^k$ is unit-regular for all $k\geq 1$.
\end{cor}

To finish this section, we discuss examples showing that Theorem \ref{Thm:PowerReg} cannot be significantly improved.  First, we prove that each of the $n$ different regularity conditions given as hypotheses in Theorem \ref{Thm:PowerReg} is independent of the others.

\begin{example}\label{Ex:RegPowerArb}
Let $I\subseteq \Z_{>0}$ be an arbitrary subset of the positive integers.  There exists a ring $R$ and an element $a\in R$ such that $a^k$ is regular in $R$ if and only if $k\in I$.
\end{example}
\begin{proof}
Let $F$ be a field, and let
\[
R:=F\langle a,\ x_i,\ (i\in I)\ :\ a^ix_i a^i=a^i\rangle.
\]
The relations $a^i x_i a^i=a^i$ (for $i\in I$) form a reduction system, and a quick verification shows that all ambiguities are resolvable in the sense of \cite{Bergman}.  Thus by Theorem 1.2 of that paper, every element of $R$ can be put into a unique reduced form by repeatedly replacing monomials on the left-hand side of the reductions with those on the right-hand side.

Clearly, the element $a^k$ is regular in $R$ if $k\in I$.  So, it suffices to prove that if $a^k$ is regular then $k\in I$.  Hence, suppose there exists some $r\in R$ with $a^k r a^k=a^k$.  As the reductions replace single monomials with single monomials, we may as well assume $r$ is a monomial, and write $r=a^{p_0}x_{i_1}a^{p_1}x_{i_2}\cdots x_{i_{n}}a^{p_n}$ for some integer $n\geq 0$, and some integers $p_0,p_1,\ldots, p_n\geq 0$.  If $n=0$, then $r$ is a power of $a$, and then $a^k ra^k\neq a^k$, so we may assume $n\geq 1$ and is minimal.  We also assume $r$ is in reduced form.

First consider the case when $n>1$.  In order for $a^k r a^k=a^k$ to hold, some reduction is necessary, and so since $r$ is reduced either $a^{k+p_0}x_{i_1}a^{p_1}$ is reducible or $a^{p_{n-1}}x_{i_n}a^{k+p_n}$ is reducible.  Without loss of generality we will assume the first is reducible, and so $i_1\leq \min(k+p_0,p_1)$ and the monomial reduces to $a^{k+p_0+p_1-i_1}$.  Hence $a^k r a^k$ reduces to
\[
a^k (a^{p_0+p_1-i_1}x_{i_2}a^{p_2}x_{i_3}\cdots x_{i_n}a^{p_n})a^k
\]
contradicting the minimality of $n$.  Thus we must have $n=1$, and so $a^{k+p_0}x_{i_1}a^{k+p_1}=a^{k}$.  This happens only if $p_0=p_1=0$ and $k=i_1\in I$, as claimed.
\end{proof}

\begin{remark}
We can modify Example \ref{Ex:RegPowerArb} so that $a^{k}$ is unit-regular for $k\in I$, and not regular for $k\in \Z_{>0}\setminus I$, by adding elements $y_i$ and relations $x_i y_i=y_i x_i=1$ to the definition of $R$.  One then shows that if $r\in R$ is a reduced monomial containing a variable $y_i$, then $a^k r a^k$ (in its reduced form) must still contain the variable $y_i$.  So it suffices to consider only those monomials $r\in R$ consisting of letters $a$ and $x_i$, and then the proof proceeds as above.
\end{remark}

It is natural to ask whether it is the case that when $a^k$ is regular in $R$ for \emph{every} $k\geq 1$, if the ``power inner inverse'' condition can be made to hold for all positive powers simultaneously (rather than up to an arbitrary, but fixed, bound).  Unfortunately, this need not be the case.

\begin{example}\label{Ex:NoUltraPower}
There exists a ring $R$ and an element $a\in R$ all of whose powers are regular in $R$, but there is no element $w\in R$ satisfying $a^{k}w^ka^k=a^k$ for every $k\geq 1$.
\end{example}
\begin{proof}
Let $F$ be a field, and let
\[
R:=F\langle a,\ x_i\ (i\in \Z_{>0})\ :\ a^i x_i a^i=a^i\rangle.
\]
Clearly, all powers of $a$ are regular.  Let $w\in R$ be arbitrary.  As $w$ is a finite sum of monomials, we must have $w\in R_n$ for some sufficiently large $n\geq 1$, where
\[
R_n:=F\langle a,\ x_i\ (1\leq i\leq n)\ :\ a^i x_i a^i=a^i\rangle.
\]
By Example \ref{Ex:RegPowerArb}, we know that $a^{n+1}$ is not regular in $R_n$, and in particular we cannot have the equality $a^{n+1}w^{n+1}a^{n+1}=a^{n+1}$, which proves the claim.
\end{proof}

Even though this latest example shows that the ``power inner inverse'' condition may not always hold for all powers simultaneously, there are some situations where this does happen.

\begin{cor}\label{Cor:powers}
Assume $a,a^2,\ldots, a^{n}\in \operatorname{reg}(R)$ for some $n\geq 1$.  If $\ann_r(a^{n-1})=\ann_r(a^{n})$, then there exists $w\in R$ such that $a^{k}w^ka^k=a^k$ for all $k\geq 1$ $($and in particular, all powers of $a\in R$ are regular$)$.  If in addition $a\in \operatorname{ureg}(R)$, then we can take $w\in U(R)$.
\end{cor}
\begin{proof}
Construct $w$ as in Theorem \ref{Thm:PowerReg}, and follow the notation given there.  Since $a^n(1-w^n a^n)=0$, we have $a^{n-1}(1-w^n a^n)=0$ and therefore
\[
a^{n-1}=a^{n-1}w^n a^n=wa^n
\]
by (\ref{powers}) with $i=n-1$ and $j=n$.  Hence, for any $k\geq n$,
\[
a^k w^k a^k =a^k w^{n-1}a^{n-1}=a^{k-n+1}(a^{n-1}w^{n-1}a^{n-1})=a^k.\qedhere
\]
\end{proof}

As a special case of this corollary we have:

\begin{cor}
If $R$ is a regular ring of bounded index of nilpotence $($so, in particular, $R$ is unit-regular$)$, then for every element $a\in R$ there exists $w\in U(R)$ such that $a^kw^ka^k=a^k$ for all $k\geq 1$.
\end{cor}

There is at least one further natural situation where elements satisfy this strong form of the power inner inverse condition, which was pointed out to us by T.Y.\ Lam in personal communication, and has the added benefit of yielding additional information about the powered inner inverses.  We thank him for allowing us to include it here.

\begin{prop}\label{Prop:LamGift}
Suppose $a\in \operatorname{ureg}(R)$.  Writing $a=eu$ with $e\in \operatorname{idem}(R)$ and $u\in U(R)$, suppose further that $eue=ue$. Then for every $k\geq 1$, the element $u^{-k}$ is an inner inverse for $a^k$.
\end{prop}
\begin{proof}
Clearly $au^{-1}a=euu^{-1}eu=e^2u=eu=a$, establishing the $k=1$ case.  Now fixing $k\geq 2$, we recursively compute
\[
a^k=(eu)^k=(eueu)(eu)^{k-2}=ueu(eu)^{k-2}=u(eu)^{k-1}=\cdots =u^{k-1}(eu).
\]
Thus $a^k u^{-k}a^{k}=u^{k-1}(eu)u^{-k}u^{k-1}(eu)=u^{k-1}(eu)=a^{k}$, as needed.
\end{proof}

\begin{remark}
Note that the elements of Proposition \ref{Prop:LamGift} also satisfy the conditions of Corollary \ref{Cor:powers}, since it happens that ${\rm ann}_r(a^k)=(1-u^{-k}a^k)R=(1-u^{-k}u^{k-1}eu)R=(1-u^{-1}eu)R$ holds for all $k\geq 1$.
\end{remark}

We finish this section with two examples showing that the power inner inverse condition for unit-regular elements does \emph{not} imply cleanness, even if we add the extremely restrictive condition that we are working in a regular ring!  So, in particular, unit-regular elements of exchange rings need not be clean.

\begin{example}\label{Example:RegPowerNotClean}
There exists a regular ring $S$, an element $a\in S$, and a unit $w\in U(S)$, such that $a^kw^ka^k=a^k$ for all $k\geq 1$, but $a$ is not clean in $S$.
\end{example}
\begin{proof}
We use an example of Bergman which can be found in \cite[Theorem 6.11]{DDN}.  Let $F$ be a field and put
\begin{align*}
R=\big\{A=(a_{i,j})\ (i,j\in \Z_{>0})\ : &\ \textnormal{ there exist $n\geq 0$ and $\psi(A)=\textstyle\sum_{k\geq k_0}a_kt^k\in F(\!(t)\!)$}\\
& \ \textnormal{ such that $a_{i,j}=a_{j-i}$ for all $i>n$ and $j\geq 1$}\big\}.
\end{align*}
These are the $\Z_{>0}\times\Z_{>0}$ column-finite matrices over $F$ which have constant diagonals outside a finite set of rows, and thus this is just the ring $R$ from \cite[Example 3.1]{SterCorner} defined in another way.  As in \cite{SterCorner}, note that the set map $\psi$ yields a homomorphism $\psi:R\to F(\!(t)\!)$.  Put $I=\ker(\psi)$, so that
\[
I=\{A=(a_{i,j})\in R\ :\ \textnormal{ there exists $n\geq 0$ such that $a_{i,j}=0$ for all $i>n$ and $j\geq 1$}\}
\]
is the set of matrices which are zero outside a finite set of rows.  From \cite{SterCorner} we know that $R$ is a regular ring which is not clean (in fact, the matrices $A\in R$ with $\psi(A)\notin F\llbracket t\rrbracket$ are not clean in $R$).

Let $S$ be the subring of $\M_2(R)$ defined as
\[
S=\left(\begin{array}{cc}R&I\\I&R\end{array}\right).
\]
Note that $S$ is regular by \cite[Lemma 1.3]{Goodearl} since it has an ideal $J=\M_2(I)$ such that $J$ and $S/J\cong F(\!(t)\!)\times F(\!(t)\!)$ are both regular. Let $\alpha\in R$ denote the right shift operator, i.e.\ $\alpha=(a_{i,j})$ where $a_{i,j}=1$ if $i-j=1$ and $a_{i,j}=0$ otherwise. Set
\[
a=\left(\begin{array}{cc}\alpha&0\\0&0\end{array}\right)\in S.
\]
The element $a$ is not clean in $S$. Indeed, suppose to the contrary that $a=e+u$ for some $e\in \operatorname{idem}(S)$ and $u\in U(S)$.  The homomorphism $\psi:R\to F(\!(t)\!)$ naturally induces a homomorphism $\psi: S\to F(\!(t)\!)\times F(\!(t)\!)$.  Write $\psi(e)=(\varepsilon_1,\varepsilon_2)$ and $\psi(u)=(\mu_1,\mu_2)$.  Note that
$\varepsilon_1,\varepsilon_2\in\{0,1\}$ as $e$ is an idempotent, and also note $\mu_1,\mu_2\neq 0$ as $u$ is a unit.  From $a=e+u$ we have
\[
(t^{-1},0)=\psi(a)=\psi(e)+\psi(u)=(\varepsilon_1,\varepsilon_2)+(\mu_1,\mu_2),
\]
so that $\mu_1=t^{-1}-\varepsilon_1$ and $\mu_2=-\varepsilon_2$.  Hence $\varepsilon_2=1$, so that either $\psi(u)=(t^{-1},-1)$ or $\psi(u)=(t^{-1}-1,-1)$.  Thinking of $S$ as consisting of infinite matrices over $F$, this means that $u$ looks like $\operatorname{diag}(\alpha,-I)$ or $\operatorname{diag}(\alpha-I,-I)$ outside of finitely many rows. A quick check shows that in none of these cases can $u$ be a surjective endomorphism of $(\bigoplus_{i\geq 1}F)^2$, and hence cannot be a unit in $S$.  This proves that $a$ is not clean in $S$.

In order to construct the inner inverse $w\in U(S)$ of $a$, denote by $\alpha'\in R$ the left shift operator, and let $\sigma=(s_{i,j})\in I$ be the matrix with $s_{i,j}=1$ if $i=j=1$ and $s_{i,j}=0$ otherwise. Set
\[
w:=\left(\begin{array}{cc}\alpha'&0\\\sigma&\alpha\end{array}\right)\in S
\quad\textnormal{and}\quad
w':=\left(\begin{array}{cc}\alpha&\sigma\\0&\alpha'\end{array}\right)\in S.
\]
Considering $\alpha'\alpha=1$ and $\alpha\alpha'=1-\sigma$, we easily check that $ww'=w'w=1$, so that $w$ is a unit in $S$. Moreover, an easy verification shows that $w^ka^k=\left(\begin{smallmatrix}1&0\\ 0&0\end{smallmatrix}\right)$ for all $k\geq 1$, which also gives $a^kw^ka^k=a^k$.
\end{proof}

This example still leaves open the possibility that by forcing certain algebraic expressions to be unit-regular, the element might become clean.  For instance, if $a,1-a\in \operatorname{ureg}(R)$, is $a$ clean in $R$?  The answer to this question is no in general.  We can prove this by generalizing the previous example.

\begin{example}
For any field $F$, there exists a regular $F$-algebra $S$ and an element $a\in S$ such that for every $p(x)\in F[x]$ it happens that $p(a)\in \operatorname{ureg}(S)$, but $a$ is not clean in $S$.
\end{example}
\begin{proof}
We assume $F$ is algebraically closed, by replacing it with its algebraic closure if necessary.  Define the ring $R$ and the ideal $I$ just as in the proof of Example \ref{Example:RegPowerNotClean}.  Similarly, let $\alpha\in R$ denote the right shift operator, $\alpha'\in R$ the left shift operator, and $\sigma\in R$ the matrix with $1$ in the upper-left corner and zeros elsewhere (just as in the previous example).  Let $\Lambda=F\cup\{\diamond\}$ and let $S$ be the set of all $\Lambda\times \Lambda$ matrices $A=(a_{i,j})_{i,j\in \Lambda}$ over $R$ with only finitely many nonzero off-diagonal entries such that all off-diagonal entries belong to $I$.

The set $S$ is a ring under the usual matrix addition and multiplication.  Further $S$ is regular because it has the ideal
\[
J=\{A=(a_{i,j})\in S\ :\ \forall\ i,j\in \Lambda,\ a_{i,j}\in I\}
\]
such that $J$ and $S/J\isom \prod_{\lambda\in \Lambda}F(\!(t)\!)$ are both regular.  (Elements of $J$ look like finite matrices over $I$, and since $\M_n(R)$ is regular, so are the elements of $J$.)

Let $a\in S$ be the diagonal matrix with $\alpha$ in the $(\diamond,\diamond)$ coordinate, and $\lambda$ in the $(\lambda,\lambda)$ coordinate for each $\lambda\in F$.  Our first order of business will be to prove that $a$ is not clean in $S$.  Suppose to the contrary that $a=e+u$ for some $e\in \operatorname{idem}(S)$ and $u\in U(S)$.  As $e$, $u$ and $u^{-1}$ all have only finitely many nonzero off-diagonal entries, this means
\[
a_0={\rm diag}(\alpha,\lambda_2,\lambda_3,\ldots,\lambda_n)=e_{0}+u_{0}
\]
for some $n\geq 2$ and some distinct elements $\lambda_2,\ldots, \lambda_n\in F$, with $e_0\in \operatorname{idem}(S_n)$ and $u_0\in U(S_n)$, where $S_n$ denotes the subring of $\M_n(R)$ with all off-diagonal entries lying in $I$.

The homomorphism $\psi:R\to F(\!(t)\!)$ naturally induces a homomorphism $\psi: S_n\to F(\!(t)\!)^n$.  Write $\psi(e_0)=(\varepsilon_1,\ldots,\varepsilon_n)$ and $\psi(u_0)=(\mu_1,\ldots,\mu_n)$.  Clearly each $\varepsilon_i\in\{0,1\}$ as $e_0$ is an idempotent, and also each $\mu_i\ne 0$ as $u_0$ is a unit.  Similarly as in the previous example, $a_0=e_0+u_0$ yields
\[
(t^{-1},\lambda_2,\ldots,\lambda_n)=\psi(a_0)=\psi(e_0)+\psi(u_0)=(\varepsilon_1,\ldots,\varepsilon_n)+(\mu_1,\ldots,\mu_n),
\]
so that $\mu_1=t^{-1}-\varepsilon_1$ and $\mu_i=\lambda_i-\varepsilon_i\in F$ if $i\geq 2$. Hence we have either $\psi(u)=(t^{-1},\mu_2,\ldots,\mu_n)$ or $\psi(u)=(t^{-1}-1,\mu_2,\ldots,\mu_n)$ where $\mu_i\in F$ if $i\geq 2$, which means that $u$ looks like $\operatorname{diag}(\alpha,\mu_2I,\ldots,\mu_nI)$ or $\operatorname{diag}(\alpha-I,\mu_2I,\ldots,\mu_nI)$ outside of finitely many rows. As before, we see that in none of these two cases $u$ can be surjective as an endomorphism of $(\bigoplus_{i\geq 1}F)^n$, and hence cannot be a unit in $S$. Therefore $a$ is not clean in $S$.

Let us prove now that $p(a)$ is unit-regular in $S$ for every $p(x)\in F[x]$.  We may assume that $p(x)$ is non-constant.  As $F$ is algebraically closed, there exists $\lambda\in F$ with $p(\lambda)=0$, so that $p(a)$ is a diagonal matrix with at least one diagonal entry zero. From this we see that it suffices to prove that
\[
z:=\left(\begin{array}{cc}p(\alpha)&0\\0&0\end{array}\right)
\]
is unit-regular in $S_2$. Write $\beta:=p(\alpha)$ and let $\beta'$ be any left inverse of $\beta$ in $R$. Define
\[
v:=\left(\begin{array}{cc}\beta'&0\\1-\beta\beta'&\beta\end{array}\right).
\]
Clearly, $v\in S_2$, and $v$ is invertible with $v^{-1}=\left(\begin{smallmatrix}\beta&1-\beta\beta'\\0&\beta'\end{smallmatrix}\right)$, and a quick verification shows that $zvz=z$. Thus $z\in\operatorname{ureg}(S_2)$ and hence $p(a)\in\operatorname{ureg}(S)$.
\end{proof}

\section{Open Question}

We end with a question motivated by the results of this paper. Example \ref{Ex:NoUltraPower} tells us that in general, even when all powers of an element are (unit-)regular, there needn't be a single element whose powers provide the inner inverses.  But perhaps this can be fixed by assuming more extensive ring-theoretic conditions, so we ask:

\begin{question}
If $R$ is a unit-regular ring and $a\in R$, is there some $w\in R$ such that for every $k\geq 1$, it happens that $w^k$ is an inner inverse for $a^k$?
\end{question}

\section*{Acknowledgements}

We thank K.\ Goodearl, D.\ Khurana, and T.\ Y.\ Lam for suggestions which led to improvements in the article.  This work was partially supported by a grant from the Simons Foundation (\#315828 to Pace Nielsen).  A part of this work was carried out during the second author's visit to Brigham Young University, and he thanks them for their hospitality and support during the visit.

\bibliographystyle{amsplain}
\bibliography{StronglyCleanUnitRegular_bib}

\providecommand{\bysame}{\leavevmode\hbox to3em{\hrulefill}\thinspace}
\providecommand{\MR}{\relax\ifhmode\unskip\space\fi MR }
\providecommand{\MRhref}[2]{%
  \href{http://www.ams.org/mathscinet-getitem?mr=#1}{#2}
}
\providecommand{\href}[2]{#2}
\begin{thebibliography}{10}

\bibitem{Ara}
Pere Ara, \emph{Strongly {$\pi$}-regular rings have stable range one}, Proc.
  Amer. Math. Soc. \textbf{124} (1996), no.~11, 3293--3298. \MR{1343679
  (97a:16024)}

\bibitem{AM}
Pere Ara and K.~C. O'{M}eara, \emph{The nilpotent regular element problem},
  Preprint, available on the arXiv (2015), 10 pp.

\bibitem{Bergman}
George~M. Bergman, \emph{The diamond lemma for ring theory}, Adv. in Math.
  \textbf{29} (1978), no.~2, 178--218. \MR{506890 (81b:16001)}

\bibitem{CDN}
Victor Camillo, Thomas~J. Dorsey, and Pace~P. Nielsen, \emph{Dedekind-finite
  strongly clean rings}, Comm. Algebra \textbf{42} (2014), no.~4, 1619--1629.
  \MR{3169656}

\bibitem{CK}
Victor Camillo and Dinesh Khurana, \emph{A characterization of unit regular
  rings}, Comm. Algebra \textbf{29} (2001), no.~5, 2293--2295. \MR{1837978}

\bibitem{CKLNZ}
Victor Camillo, Dinesh Khurana, T.~Y. Lam, W.~K. Nicholson, and Y.~Zhou,
  \emph{Continuous modules are clean}, J. Algebra \textbf{304} (2006), no.~1,
  94--111. \MR{2255822 (2007d:16065)}

\bibitem{CamilloNielsen}
Victor Camillo and Pace~P. Nielsen, \emph{Mc{C}oy rings and zero-divisors}, J.
  Pure Appl. Algebra \textbf{212} (2008), no.~3, 599--615. \MR{2365335
  (2008k:16056)}

\bibitem{CamilloYu}
Victor Camillo and Hua-Ping Yu, \emph{Exchange rings, units and idempotents},
  Comm. Algebra \textbf{22} (1994), no.~12, 4737--4749. \MR{1285703
  (95d:16013)}

\bibitem{CJ}
Peter Crawley and Bjarni J{\'o}nsson, \emph{Direct decompositions of algebraic
  systems}, Bull. Amer. Math. Soc. \textbf{69} (1963), 541--547. \MR{0156808
  (28 \#52)}

\bibitem{DDN}
Alexander~J. Diesl, Samuel~J. Dittmer, and Pace~P. Nielsen, \emph{Idempotent
  lifting and ring extensions}, J. Algebra Appl. (to appear 2016).

\bibitem{EhrlichFirst}
Gertrude Ehrlich, \emph{Unit-regular rings}, Portugal. Math. \textbf{27}
  (1968), 209--212. \MR{0266962 (42 \#1864)}

\bibitem{Ehrlich}
\bysame, \emph{Units and one-sided units in regular rings}, Trans. Amer. Math.
  Soc. \textbf{216} (1976), 81--90. \MR{0387340 (52 \#8183)}

\bibitem{Goodearl}
K.~R. Goodearl, \emph{Von {N}eumann {R}egular {R}ings}, second ed., Robert E.
  Krieger Publishing Co., Inc., Malabar, FL, 1991. \MR{1150975 (93m:16006)}

\bibitem{HO}
John Hannah and K.~C. O'Meara, \emph{Products of idempotents in regular rings.
  {II}}, J. Algebra \textbf{123} (1989), no.~1, 223--239. \MR{1000485
  (90h:16022)}

\bibitem{Khurana}
Dinesh Khurana, \emph{Unit-regularity of regular nilpotent elements}, Preprint,
  available on the arXiv (2015), 3 pp.

\bibitem{KL}
Dinesh Khurana and T.~Y. Lam, \emph{Clean matrices and unit-regular matrices},
  J. Algebra \textbf{280} (2004), no.~2, 683--698. \MR{2090058 (2005f:16014)}

\bibitem{LamExercises}
T.~Y. Lam, \emph{Exercises in {C}lassical {R}ing {T}heory}, second ed., Problem
  Books in Mathematics, Springer-Verlag, New York, 2003. \MR{2003255
  (2004g:16001)}

\bibitem{Nicholson77}
W.~K. Nicholson, \emph{Lifting idempotents and exchange rings}, Trans. Amer.
  Math. Soc. \textbf{229} (1977), 269--278. \MR{0439876 (55 \#12757)}

\bibitem{NicholsonStrongClean}
\bysame, \emph{Strongly clean rings and {F}itting's lemma}, Comm. Algebra
  \textbf{27} (1999), no.~8, 3583--3592. \MR{1699586 (2000d:16046)}

\bibitem{SterCorner}
Janez {\v{S}}ter, \emph{Corner rings of a clean ring need not be clean}, Comm.
  Algebra \textbf{40} (2012), no.~5, 1595--1604. \MR{2924469}

\bibitem{WCKL}
Zhou Wang, Jianlong Chen, Dinesh Khurana, and Tsit-Yuen Lam, \emph{Rings of
  idempotent stable range one}, Algebr. Represent. Theory \textbf{15} (2012),
  no.~1, 195--200. \MR{2872487}

\bibitem{Warfield}
R.~B. Warfield, Jr., \emph{Exchange rings and decompositions of modules}, Math.
  Ann. \textbf{199} (1972), 31--36. \MR{0332893 (48 \#11218)}

\bibitem{Zhang}
Hongbo Zhang, \emph{On strongly clean modules}, Comm. Algebra \textbf{37}
  (2009), no.~4, 1420--1427. \MR{2510991 (2010a:16004)}

\end{thebibliography}

\end{document}